\documentclass[final,3p,times,square,sort&compress]{astyle}
\usepackage{txfonts}
\usepackage{amssymb}
\usepackage{CJKutf8} % 注意要用 CJKutf8 而不是 CJK 做中文书签
\usepackage{amsmath} %%联立符号%% The amsthm package provides extended theorem environments
\usepackage[colorlinks,  %此步不可省略
            linkcolor=blue,
            anchorcolor=red,
            unicode={true},
            citecolor=green]{hyperref}
\newtheorem{thm}{Theorem}[section]
\newproof{proof}{Proof}

\begin{document}
\begin{frontmatter}

%% \title{Title\tnoteref{label1}}
%% \tnotetext[label1]{}
%% \author{Name\corref{cor1}\fnref{label2}}
%% \ead{email address}
%% \ead[url]{home page}
%% \fntext[label2]{}
%% \cortext[cor1]{}
%% \address{Address\fnref{label3}}
%% \fntext[label3]{}
 
\title{The General Solutions of Linear ODE and Riccati Equation  by Integral Series\tnoteref{label1}
}
\tnotetext[label1]{Many thanks to Prof.Qiyan Shi's guidance.}

\author{Yimin Yan
}
\ead{yanyimin@foxmail.com}
%% use optional labels to link authors explicitly to addresses:
%% \author[label1,label2]{<author name>}

%% \address[label2]{<address>}
\address{}

\begin{abstract}
This paper gives out the general solutions of variable coefficients Linear ODE and Riccati equation  by way of integral series $\mathcal{E }(X )$ and $\mathcal{F}(X )$. Such kinds  of integral series are the generalized form of exponential function, and keep the properties of  convergent and reversible.
\end{abstract}
\begin{keyword}
%% keywords here, in the form: keyword \sep keyword

%% MSC codes here, in the form: \MSC code \sep code
%% or \MSC[2008] code \sep code (2000 is the default)
Linear ODE,Riccati equation,integral series, general solution,variable coefficients
\end{keyword}
\end{frontmatter}

% \linenumbers

%% main text
 \section{Introduction}
It is a classical problem to solve
   the n-th   order Linear ODE :
\begin{equation}\label{LODE.1}
    {\frac {d^{n}}{d{x}^{n}}}u+a_{1}(x){\frac {d^{n-1}}{d{x}^{n-1}}}u+a_{2}(x){\frac {d^{n-2}}{d{x}^{n-2}}}u+\cdots+a_{n}(x) u=f(x)
\end{equation}
which is equivalent to

\begin{equation}\label{LODE}
    {\frac {d}{dx}} U=AU+F
\end{equation}
with
    \begin{equation}
      \left\{
       \begin{aligned}
            U&={ \left[ \begin {array}{cccc}
                {\frac {d^{n-1}}{d{x}^{n-1}}}u   &{\frac {d^{n-2}}{d{x}^{n-2}}}u   &\cdots &u \end {array} \right]}^T\\
            F&=\left[ \begin {array}{cccc} f \left( x \right) &0&\cdots&0\end {array} \right]^T\\
            A(x)&= \left[ \begin {array}{ccccc}
                   -a_{{1}}&-a_{{2}}&-a_{{3}}&\cdots&-a_{{n}}\\
                \noalign{\medskip}1&0&0&\cdots&0\\
                 \noalign{\medskip}0&1&0&\cdots&0\\
                \noalign{\medskip}\cdots  &\cdots&\cdots& \cdots&\cdots\\
                \noalign{\medskip}0&0&\cdots&1&0\end {array} \right]\\
       \end{aligned}
       \right.
  \end{equation}
\\
As we all known,
\begin{enumerate}
  \item {if $\big\{ a_{n}(x)\big\}$ are all constants,  Eq.(\ref{LODE.1}) could be  solved by method of eigenvalue ( Euler), or by  exponential function in matrix form
      \begin{equation*}
        U= e^{ A\cdot x}\cdot C+ e^{ A\cdot x}\cdot \int _{0}^{x}e^{ -A\cdot s}\cdot F \left( s \right) {ds}
      \end{equation*}
    \emph{  where C     is a $n\times 1 $ constant matrix .}
  }
  \item {if $\big\{ a_{n}(x)\big\}$ are some variable coefficients, such as some special functions \cite[P337,206]{Wang}

  \begin{equation*}\tag{Bessel Equation}
     \frac{d^2 y}{dx^2} +\frac{1}{x} \frac{dy}{dx}+\big(1-\frac{n^2}{x^2}\big)y=0
  \end{equation*}

  \begin{equation*}\tag{Legendre Equation}
    (1-x^2)\frac{d^2 y}{dx^2}-2x\frac{dy}{dx}+n(n+1)y=0
  \end{equation*}
   special function theory answers them.   }
\end{enumerate}
But when it comes to the general circumstances,
the existing methods meet difficulties  in dealing with  Eq.(\ref{LODE}) , because of the variable coefficients.
In order to overcome it,   two functions are invited :
\subsection{Definition}
    \begin{equation}\label{intro.e}
      \left\{
       \begin{aligned}
            \mathcal{E}\big [X(x)\big ]=&I+\int _{0}^{x}\!X \left( t \right) {dt}+\int _{0}^{x}\!X \left( t
             \right) \int _{0}^{t}\!X \left( s \right) {ds}{dt} +\int _{0}^{x}\!X
             \left( t \right) \int _{0}^{t}\!X \left( s \right) \int _{0}^{s}\!X
             \left( \xi \right) {d\xi}{ds}{dt}+\cdots\\
           \mathcal{F}\big [X(x)\big ]=&I+\int _{0}^{x}\!X \left( t \right) {dt}+\int _{0}^{x}\! \int _{0}^{t}
            \!X \left( s \right) {ds} X \left( t \right) {dt} +\int _{0}^{x}\!
            \int _{0}^{t}\! \int _{0}^{s}\!X \left( \xi \right)
            {d\xi}\ X ( s ){ds}\ X \left( t \right) {dt}+\cdots\\
       \end{aligned}
       \right.
  \end{equation}
It will be  seen that such definition is reasonable and necessary. Clearly, when $X(x)$ and $\int _{0}^{x}\!X(t)dt$  are exchangeable, then
\begin{equation*}
    \mathcal{E}\big [X(x)\big ]=e^{\int _{0}^{x}\!X(t)dt}=\mathcal{F}\big [X(x)\big ]
\end{equation*}

 Besides, $\mathcal{E}( X )$ and $\mathcal{F}( X )$ extend some main properties of the exponential functions,  such as convergent , reversible and determinant (see Theorem \ref{th.property}). In addition,  a $n\times m$ matrix $A(x)=\big(a_{ij}(x)\big)_{nm}$
 is bounded and integral in [0,b] means that all its element $a_{ij}(x)$ are bounded and integral in [0,b].
%%%%%%%%%%%%%%%%%%%%%%%%%%%%%%%%%%%%%%%%%%%%%%%%%%%%
%%%%%%%%%%%%%%%%%%%%%%%%%%%%%%%%%
%%%%%%%%%%%%%%%%%%%%%
\section{ Main Results }
\begin{thm}\label{th.ODE}
the general solution of the  Linear ODE (\ref{LODE}) is:
\begin{equation}
    U=\mathcal{E}\big [A(x)\big ]\cdot C+ \mathcal{E}\big [A(x)\big ]\cdot \int _{0}^{x} \mathcal{F}\big [-A( s)\big ]\cdot F \left( s \right) {ds}
\end{equation}
where C     is a $n\times 1 $ constant matrix .
\end{thm}

%%%%%%%%%%%%%%%%%%%%%%%%%%%%%%%%%%%%%%%%%%%%%%%%%%%%%%%%%%%%%%%%%%%%%%%%%%%%%%%%%%%%%%%%
\begin{thm}\label{th.Riccati}
For the bounded and integrable matrix , $A(x)=(a_{ij})_{nn} $, $B(x)=(b_{ij})_{mm} $,  $P(x)=(p_{ij})_{mn} $, $Q(x)=(q_{ij})_{nm} $, in [0,b],  the  general solution of  Riccati equation
\begin{equation}
    {\frac {d}{dx}}W+WPW+WB-AW-Q=0
\end{equation}
 is
 \begin{equation}
    W=W_{1} \cdot W^{-1}_{2}
\end{equation}
where
\begin{equation}
     \left[ \begin {array}{c} W_{{1}}\\ \noalign{\medskip}W_{{2}}\end {array} \right]
    =\mathcal{E}\biggl(\left[ \begin {array}{cc} A&Q\\ \noalign{\medskip}P&B\end {array}
 \right] \biggl)\cdot\left[ \begin {array}{c} W\mid_{x=0}\\ \noalign{\medskip}I\end {array} \right]
\end{equation}
or the other equivalent form:
 \begin{equation}
    W=U^{-1}_{2} \cdot U_{1}
\end{equation}
where
\begin{equation}
    \begin{array}{ll}
        \left[ \begin {array}{cc} U_{1}&U_{2}\end {array} \right]
        =\left[ \begin {array}{cc} I& W\mid_{x=0}\end {array} \right]\cdot \mathcal{F}\biggl(\left[ \begin {array}{cc} -B&P\\ \noalign{\medskip}Q&-A\end {array} \right]\biggl)
    \end{array}
\end{equation}

\end{thm}

%%%%%%%%%%%%%%%%%%%%%%%%%%%%%%%%%%%%%%%%%%%%%
%%%%%%%%%%%%%%%%%%%%%%%%%%%%%%%%%%%%%%%%%%%%%

%%%%%%%%%%%%%%%%%%%%%%%%%%%%%%%%%%%%%%%%%%%%%%%%%%%%%%%%%%%%%%%%%%%%%%%%%%%%%
\section{Solutions of Linear ODE  }

\subsection{ Properties of $\mathcal{E}( X )$ and $\mathcal{F}( X )$}

From the  Definition(\ref{intro.e}), it holds that
    \begin{equation}
      \left\{
       \begin{aligned}
        {\frac {d}{dx}}\mathcal{E}\biggl[ X(x)\biggl]&=X\cdot \mathcal{E}\biggl[ X(x) \biggl]\\
        {\frac {d}{dx}}\mathcal{F}\biggl[ X(x)\biggl]&= \mathcal{F}\biggl[ X(x) \biggl]\cdot X\\
       \end{aligned}
       \right.
  \end{equation}

Now, we will see more explicit properties of $\mathcal{E}( X )$ and $\mathcal{F}( X )$.

%%%%%%%%%%%%%%%%%%%%%%%%%%%%%%%%%%%%%%%%%%%%%%%%%%%%%%%%%%%%%%%%%%%%%%%%%%%%%%%%%%%%%%%%%%%
\begin{thm}[Properties of $\mathcal{E}( X )$ and $\mathcal{F}( X )$]    \label{th.property}
If $X(x)$ is bounded and integrable, it holds that
\begin{enumerate}
  \item {$\mathcal{E}( X )$ and $\mathcal{F}( X )$ are convergent;
  }
  \item {
  \begin{equation} \label{det}
  \det\mathcal{E}( X ) = \det\mathcal{F}(X)= \det e^{\int _{0}^{x}\!X(t) {dt}}
       =e^{\int _{0}^{x}\!tr X(t) {dt}}=e^{tr \int _{0}^{x}\!X(t) {dt}}
    \end{equation}
   }
  \item {$\mathcal{E}( X )$ and $\mathcal{F}( X )$ are reversible, and
    \begin{equation}
        \mathcal{F}(X)\mathcal{E}(-X)=\mathcal{E}(-X)\mathcal{F}(X)=I
    \end{equation}
  }
\end{enumerate}
\end{thm}
%%%%%%%%%%%%%%%%%%%%%%%%%%%%%%%%%%%%%%%%%%%%%%%%%%%%%%%%%%%%%%%%%%%%%%%%%%%%%%%%%%%%%%%%%%%%%
\begin{proof}
\begin{enumerate}
  \item {Firstly, $\mathcal{E}(A)$ is convergent,since $\bigl \{a_{k}(x) \bigl \}^{n}_{k=1}$ are bounded in [0,b]:
\begin{center}
    $\exists M>0$, s.t. $|a_{k}(x)|< M$,\qquad  $\forall  x \in [0,b]$, \qquad $k=1,2,\cdots,n$
\end{center}
So
\begin{enumerate}
  \item {
  \begin{eqnarray*}
  % \nonumber to remove numbering (before each equation)
     \big \| \int _{0}^{x}\!A \left( t \right) {dt} \big \| =max \left| \int _{0}^{x}\!a_{{k}} \left( t \right) {dt} \right|<M |x|
  \end{eqnarray*}
  }
  \item {
  \begin{eqnarray*}
     &     \big  \| \int _{0}^{x}\!A \left( t \right) \int _{0}^{t}\!A \left( s \right) {ds}{dt} \big \|
      =max\left| \sum_{i}\int _{0}^{x}\!a_{{k}} \left( t \right) \int _{0}^{t}\!a_{{i}} \left( s \right) {ds}{dt} \right|
      <n M^2\left|   \int _{0}^{x}\! \int _{0}^{t}\! 1 {ds}{dt} \right|
 <\frac{n}{2!}(M|x|)^2
   \end{eqnarray*}
  }
  \item {
    \begin{eqnarray*}
     &&       \big \| \int _{0}^{x}\!A \left( t \right) \int _{0}^{t}\!A \left( s
         \right) \int _{0}^{s}\!A \left( \xi \right) {d\xi}{ds}{dt} \big \|
         =max \left| \sum_{i,j}\int _{0}^{x}\!a_{{k}} \left( t \right) \int _{0}^{t}\!a_{{i}}
         \left( s \right) \int _{0}^{s}\!a_{{j}} \left( \xi \right) {d\xi}{ds}
        {dt} \right|\\
        <&&n^2 M^3\left| \int _{0}^{x}\! \int _{0}^{t}\!\int _{0}^{s}\!{d\xi}{ds}{dt} \right|
        <\frac{n^2}{3!} (M|x|)^3
  \end{eqnarray*}

  }
  \item {$\cdots\cdots$}
\end{enumerate}
It follows that
\begin{eqnarray*}
        \|\mathcal{E}(A)\|&<&1+\frac{1}{n}\biggl[nM |x|+\frac{1}{2!}(nMx)^2+\frac{1}{3!} (nMx)^3+\cdots \biggl]
        = 1+\frac{1}{n}e^{nM|x|}
\end{eqnarray*}
 \\

Clearly, $\mathcal{E}(A)$ is convergent.  \\

  Similarly, $\mathcal{F}( X )$ is also convergent.
  }
  \item {

    \emph{$\forall n\times n$ matrix $A(x)$, if $tr A(x)$ is bounded and integral , then }
    \begin{equation}
        \det\mathcal{E}( A(x) )=e^{\int _{0}^{x}\!tr A(t) {dt}}=e^{tr \int _{0}^{x}\!A(t) {dt}}
    \end{equation}

which  is a special case of   Abel's formula\cite{Chen}:
\emph{If W and B are $ n\times n$ matrixes , s.t.}
    \begin{equation}\label{able}
        {\frac {d}{dx}}W=BW
    \end{equation}
then,
    \begin{equation}
        \det W=e^{tr B}
    \end{equation}\\

Here we just take $2\times 2$ matrix for   verification:\\
Let
$Y(x)=\mathcal{E}\biggl[A(x)\biggl]= \left[ \begin {array}{cc} y_{{1,1}}&y_{{1,2}}\\ \noalign{\medskip}y_{
{2,1}}&y_{{2,2}}\end {array} \right]$,  $A(x)=\left[ \begin {array}{cc} a_{{1,1}}&a_{{1,2}}\\ \noalign{\medskip}a_{
{2,1}}&a_{{2,2}}\end {array} \right]$\\
so  $ {\frac {d}{dx}}Y=A\cdot Y$means that
    \begin{equation}
        {\frac {d}{dx}}\left[ \begin {array}{cc} y_{{1,1}}&y_{{1,2}}\\ \noalign{\medskip}y_{
            {2,1}}&y_{{2,2}}\end {array} \right]
                = \left[ \begin {array}{cc} a_{{1,1}}&a_{{1,2}}\\ \noalign{\medskip}a_{
            {2,1}}&a_{{2,2}}\end {array} \right]\cdot   \left[ \begin {array}{cc} y_{{1,1}}&y_{{1,2}}\\ \noalign{\medskip}y_{
            {2,1}}&y_{{2,2}}\end {array} \right]
    \end{equation}
 it follows

    \begin{equation*}
       \begin{aligned}
        {\frac {d}{dx}}( \det Y)
        &= \det\left[ \begin {array}{cc}{\frac {d}{dx}} y_{{1,1}}&{\frac {d}{dx}}y_{{1,2}}\\
                    \noalign{\medskip}y_{{2,1}}&y_{{2,2}}\end {array} \right]
             +\det\left[ \begin {array}{cc} y_{{1,1}}&y_{{1,2}}\\ \noalign{\medskip}{\frac {d}{dx}}y_{{2,1}}&{\frac {d}{dx}}y_{{2,2}}\end {array} \right] \\
       & = det\left[ \begin {array}{cc}{ a_{1,1} y_{1,1}+a_{1,2}y_{2,1}}&{a_{1,1} y_{1,2}+a_{1,2}y_{2,2}}\\ \noalign{\medskip}y_{{2,1}}&y_{{2,2}}\end {array} \right]
                +\det\left[ \begin {array}{cc} y_{{1,1}}&y_{{1,2}}\\ \noalign{\medskip}  a_{2,1} y_{1,1}+a_{2,2}y_{2,1}& a_{2,1} y_{1,2}+a_{2,2}y_{2,2} \end {array} \right]\\
       & =
                a_{1,1}\det\left[ \begin {array}{cc} y_{{1,1}}&y_{{1,2}}\\ \noalign{\medskip}y_{{2,1}}&y_{{2,2}}\end {array} \right]+
                a_{2,2}\det\left[ \begin {array}{cc} y_{{1,1}}&y_{{1,2}}\\ \noalign{\medskip}y_{{2,1}}&y_{{2,2}}\end {array} \right]\\
        & =\biggl[a_{1,1}+ a_{2,2}\biggl]det Y
                =tr A\cdot det Y
       \end{aligned}
   \end{equation*}

Thus, Abel's formula holds and $\mathcal{E}( A(x) )$ is reversible. \\

  By the times:
    \begin{equation}
        \det\mathcal{F}(X)=e^{\int _{0}^{x}\!tr X(t) {dt}}=e^{tr \int _{0}^{x}\!X(t) {dt}}
    \end{equation}
     so, all we need to proof is
        \begin{equation}
             \det e^{\int _{0}^{x}\!X(t) {dt}}=e^{\int _{0}^{x}\!tr X(t) {dt}}
        \end{equation}
        Because   $e^{\int _{0}^{x}\!X(t) {dt}}$ no longer satisfies Abel's formula (one reason is $X $ and  $\int _{0}^{x}\!X(t) {dt}$  are unnecessarily exchangeable ) , we seek the other approach:\\

         $\forall n\times n$ matrix A, $\exists n\times n$ reversible matrix P , s.t.
    \begin{center}
     $P^{-1} A P=diag\{J_{1},J_{2},\cdots,J_{s} \}:=J$
    \end{center}
   \emph{ J is  A's Jordan matrix, $J_{i}$ is the Jordan block with eigenvalue $\lambda_{i}(x)$.}\\
   It follows that
   \begin{equation}
        e^{J_{i}}=e^{\lambda_{i}(x)}
            \left[ \begin {array}{cccccc} 1&1&\frac{1}{2!}&\frac{1}{3!}&\cdots&\cdots\\
             \noalign{\medskip}0&1&1&\frac{1}{2!}&\cdots&\cdots\\
             \noalign{\medskip}0&0&1&1&\cdots&\cdots\\
              \noalign{\medskip}\cdots&\cdots&\cdots&\cdots&\cdots&\cdots\\
              \noalign{\medskip}0&0&0&0&\cdots&1\end {array} \right]
   \end{equation}
  So,
  \begin{center}
    $P^{-1}e^{A}P=e^{P^{-1}AP}=e^{J}=diag\{e^{J_{1}},e^{J_{2}},\cdots,e^{J_{s}} \}$
  \end{center}
  Therefore
  \begin{center}
    $ \det e^{A}=\det e^{J}= e^{tr J}=e^{tr A}$
  \end{center}
  which yields
  \begin{center}
    $\det e^{\int _{0}^{x}\!X(t) {dt}}=e^{\int _{0}^{x}\!tr X(t) {dt}}$
  \end{center}
   }
  \item {
  \emph{Notice that $\forall n\times n$ matrix A, there exists a companion matrix $A^{*}$,s.t.
  \begin{equation}       \label{det.com}
    A\cdot A^{*}=A^{*}\cdot A=\texttt{det}A \cdot I
  \end{equation}
  so, if  $\det A\neq 0$ , A is invertible.\\}

Therefore, $\mathcal{E}( X )$ and $\mathcal{F}( X )$ are invertible.\\
Furthermore, it holds that
\begin{equation}  \label{ef}
    \mathcal{F}(X)\mathcal{E}(-X)=\mathcal{E}(-X)\mathcal{F}(X)=I
\end{equation}
Because:
\begin{enumerate}
  \item {
  	\begin{equation*}
       \begin{aligned}
        	&{\frac {d}{dx}} \biggl[\mathcal{F}(X)\mathcal{E}(-X)\biggl]
        ={\frac {d}{dx}}\mathcal{F}(X)\cdot\mathcal{E}(-X)
                        +\mathcal{F}(X)\cdot{\frac {d}{dx}}\mathcal{E}(-X)
        =  \mathcal{F}(X)X\cdot\mathcal{E}(-X)-\mathcal{F}(X)\cdot X\mathcal{E}(-X)
                    = 0
       \end{aligned}
   \end{equation*}

    So,
    \begin{equation*}
       \begin{aligned}
        	\mathcal{F}(X)\mathcal{E}(-X)&=const.
         =\big [\mathcal{F}(X)\mathcal{E}(-X)\big ]\big |_{x=0}
         =I
       \end{aligned}
   \end{equation*}

  }
  \item {Due to the special property(\ref{det.com}) of matrix, Eq.(\ref{ef}) is obtained. }
\end{enumerate}
  }
\end{enumerate}
\end{proof}
%%%%%%%%%%%%%%%%%%%%%%%%%%%%%%%%%%%%%%%%%%%%%%%%%%%%%%%%%%%%%%%%%%%%%%%%%%%%%%%%%%%%%%%%%%%%%%%%%%

\subsection{ Proof of Theorem\ref{th.ODE} }

\begin{proof}
  According to  Definition(\ref{intro.e}) and Theorem \ref{th.property}  , it follows
    \begin{equation}\label{pf.ODE}
      \left\{
       \begin{aligned}
        {\frac {d}{dx}}\mathcal{E}\big [A(x)\big ] &=& A(x)\cdot\mathcal{E}\big [A(x)\big ]\\
        {\frac {d}{dx}}G(x) &=& A(x)\cdot G(x) + F\\
       \end{aligned}
       \right.
  \end{equation}
\emph{where }$$G(x)=\mathcal{E}\big [A(x)\big ]\cdot \int _{0}^{x} \mathcal{F}\big [-A( s)\big ]\cdot F \left( s \right) {ds}$$  because
	\begin{equation*}
       \begin{aligned}
            {\frac {d}{dx}}G(x)
           = & {\frac {d}{dx}}\mathcal{E}\big [A(x)\big ]\cdot \int _{0}^{x} \mathcal{F}\big [-A( s)\big ]\cdot F \left( s \right) {ds}
              +\mathcal{E}\big [A(x)\big ]\cdot   \mathcal{F}\big [-A(x)\big ]\cdot F(x) \\
           =& A(x)\cdot \mathcal{E}\big [A(x)\big ]\cdot \int _{0}^{x} \mathcal{F}\big [-A( s)\big ]\cdot F \left( s \right) {ds}+ F
       \end{aligned}
   \end{equation*}

Clearly   $U(x)=\mathcal{E}\big [A(x)\big ]\cdot C+ \mathcal{E}\big [A(x)\big ]\cdot \int _{0}^{x} \mathcal{F}\big [-A( s)\big ]\cdot F \left( s \right) {ds}$ is   convergent.\\

Moreover,   since $\mathcal{E}(A)$ is reversible, $ U(x)$ is the general solution of Eq.(\ref{LODE}).

\end{proof}
%%%%%%%%%%%%%%%%%%%%%%%%%%%%%%%%%%%%%%%%%%%%%%

\begin{thm}
Assume that $A(x)=(a_{ij})_{n\times n}$, $B(x)=(b_{ij})_{m\times m}$, $P(x)=(p_{ij})_{n\times m}$  are bounded and integrable matrixes , and $U(x)$ is  the desired  $n\times m$ matrix.  The  Linear ODE :
\begin{equation}    \label{general.ODE}
    {\frac {d}{dx}}U=A(x)U+UB(x)+P(x)
\end{equation}
has general solutions
\begin{equation}
    U(x)=\mathcal{E}(A)\biggl[ \int _{0}^{x}\! \mathcal{F}\big(-A(t)\big)P(t) \mathcal{E}\big(-B(t)\big) {dt}  +C\biggl]\mathcal{F}(B)
\end{equation}
\emph{where C is $n\times m$ constant matrix.}
\end{thm}

\begin{proof}
Let $U=\mathcal{E}(A)\cdot W \cdot\mathcal{F}(B)$, then
\begin{equation}
    {\frac {d}{dx}}U=A(x)U+UB(x)+\mathcal{E}(A) {\frac {d}{dx}}W \cdot\mathcal{F}(B)
\end{equation}

So  Eq.({\ref{general.ODE}}) could be  reduced to
\begin{equation}
    \mathcal{E}(A) {\frac {d}{dx}}W \cdot\mathcal{F}(B)=P
\end{equation}
or,
\begin{equation}
    {\frac {d}{dx}}W =\mathcal{F}(-A)\cdot P\cdot\mathcal{E}(-B)
\end{equation}

It's obviously that
\begin{equation}
    W(x)= \int _{0}^{x}\! \mathcal{F}\big [-A(t)\big ]P(t)\cdot \mathcal{E}\big [-B(t)\big ] {dt}  +C
\end{equation}
\emph{C is $n \times m$ constant matrix .}
\end{proof}

%%%%%%%%%%%%%%%%%%%%%%%%%%%%%%%%%%%%%%%%%%%%%%%%%

%%%%%%%%%%%%%%%%%%%%%%%%%%%%%%%%%%%%%%%%%%%%%%%%%

\section{Solutions of  Riccati equation}
In  mathematical investigation of the dynamics of a system, the introduction of a nonlinearity always leads to some form of the Riccati equation \cite{SolutionOfRiccati}:
\begin{equation}
    {\frac {d}{dx}}y+a(x)y^2+b(x)y+c(x)=0
\end{equation}
But it is usually the case that not even one solution of the Riccati equation is known.
 In the  following text, we try to give out solutions of  Riccati equation in matrix form:
\begin{equation} \label{r.g}
    {\frac {d}{dx}}W+WPW+WB-AW-Q=0
\end{equation}
\emph{where $A(x)=(a_{ij})_{nn} $, $B(x)=(b_{ij})_{mm} $,  $P(x)=(p_{ij})_{mn} $, $Q(x)=(q_{ij})_{nm} $ }.

\subsection{Proof of Theorem.\ref{th.Riccati}}

\begin{proof}
\begin{enumerate}
  \item {Firstly , define\cite[Ch 0.1.4]{Riccati}
    \begin{equation} \label{r.d1}
          W_{2}:=\mathcal{E}(PW+B)
    \end{equation}
    so $W_{2}$ is reversible, if $PW + B $ is bounded;\\
     meanwhile,
    \begin{equation}    \label{r.r}
        {\frac {d}{dx}} W_{2}= (PW+B)W_{2}
    \end{equation}

    Secondly, let $W_{1}:=WW_{2}$, so
    \begin{equation}
    \begin{array}{ll}
        {\frac {d}{dx}} W_{1}&={\frac {d}{dx}}W\cdot W_{2}+W\cdot {\frac {d}{dx}}W_{2}
        ={\frac {d}{dx}}W\cdot W_{2}+W\cdot \biggl[PW+B\biggl]W_{2}
        =\biggl[{\frac {d}{dx}}W+WPW+WB\biggl]W_{2}
    \end{array}
    \end{equation}

    so, with Eq.(\ref{r.g}) and Definition (\ref{r.d1}), it holds
    \begin{equation}\label{r.w2}
          {\frac {d}{dx}} W_{1}=AW_{1}+QW_{2}
    \end{equation}\\

    Take the relationship (\ref{r.r}) and (\ref{r.w2}) into consideration,
    \begin{equation}    \label{r.01}
         {\frac {d}{dx}}  \left[ \begin {array}{c} W_{{1}}\\ \noalign{\medskip}W_{{2}}\end {array} \right]
    =  \left[ \begin {array}{cc} A&Q\\ \noalign{\medskip}P&B\end {array}
    \right]  \cdot  \left[ \begin {array}{c} W_{{1}}\\ \noalign{\medskip}W_{{2}}\end {array} \right]
    \end{equation}
    we can solve $W_{1}$ and $W_{2}$.\\

    On the other hand, according to Definition (\ref{r.d1}) , it's obviously that
    \begin{equation}
            W_{2}|_{x=0}=\mathcal{E}(PW+B)|_{x=0}=I
    \end{equation}
    so it goes without saying that
    \begin{equation}
           W_{1}|_{x=0}=(W W_{2})|_{x=0}=W  |_{x=0}
    \end{equation}\\

    We immediately obtain
    \begin{equation}\label{r.uniquessEstimate}
     \left[ \begin {array}{c} W_{{1}}\\ \noalign{\medskip}W_{{2}}\end {array} \right]
    =\mathcal{E}\biggl(\left[ \begin {array}{cc} A&Q\\ \noalign{\medskip}P&B\end {array}
 \right] \biggl)\cdot\left[ \begin {array}{c} W\mid_{x=0}\\ \noalign{\medskip}I\end {array} \right]
    \end{equation}

    Therefore $W=W_{1} \cdot W^{-1}_{2}$  is the solution of Eq.(\ref{r.g}).
  }
  \item { Similarly, we can get
\begin{equation}        \label{r.02}
    \begin{array}{ll}
        {\frac {d}{dx}} \left[ \begin {array}{cc} U_{1}&U_{2}\end {array} \right]
        =\left[ \begin {array}{cc} I& W\mid_{x=0}\end {array} \right]\cdot  \left[ \begin {array}{cc} -B&P\\ \noalign{\medskip}Q&-A\end {array} \right]
    \end{array}
\end{equation}
  so,  $W=U^{-1}_{2} \cdot U_{1}$ is also the solution of Eq.(\ref{r.g}).
  }
  \item { But the two solutions are equivalence! That is,
  \begin{equation}
     W_{1} \cdot W^{-1}_{2}=U^{-1}_{2}U_{1}
   \end{equation}
   or
   \begin{equation}
     U_{2} \cdot W_{1}-U_{1}\cdot W_{2}=0
   \end{equation}

   Because, according to Eq.(\ref{r.01}) and Eq.(\ref{r.02})
\begin{equation}        \label{r.2}
    \begin{array}{ll}
        &{\frac {d}{dx}} \biggl[ U_{2} \cdot W_{1}-U_{1}\cdot W_{2}\biggl]
        = {\frac {d}{dx}} U_{2} \cdot W_{1}+U_{2} \cdot {\frac {d}{dx}}W_{1}
         -{\frac {d}{dx}}U_{1}\cdot W_{2} - U_{1}\cdot {\frac {d}{dx}} W_{2}\\
        =& \biggl[ U_{1} P-U_{2}A\biggl] \cdot W_{1}
            + U_{2} \cdot \biggl[ A W_{1}+ Q W_{2}\biggl]
            -\biggl[ U_{2} Q-U_{1}B\biggl]\cdot W_{2}
             -U_{1}\cdot \biggl[ P W_{1}+B W_{2}\biggl]
         =0
    \end{array}
\end{equation}\\

As a result,
   \begin{equation}
       \begin{aligned}
    	  U_{2} \cdot W_{1}-U_{1}\cdot W_{2}
            = const.
            = \big [U_{2} \cdot W_{1}-U_{1}\cdot W_{2}\big ]\big |_{x=0}=0
       \end{aligned}
   \end{equation}

    which implied that two solutions are equivalence.
}
\item{Uniqueness.
If Eq.(\ref{r.g}) has  more than one solution,such as  $X(x), Y(x)$,  under the same initial condition,i.e. $X(0)= Y(0)$.   Let
$W(x)=X(x)-Y(x)   $. So it is  clear that what we need to prove is equitant to show
    \begin{equation}\label{R.U.eq}
      \left\{
       \begin{aligned}
             &{\frac {d}{dx}}W+WPW+WB-AW =0\\
             &W|_{x=0}=0
       \end{aligned}
       \right.
   \end{equation}
has uniqueness solution $W(x)=0$.\\

Take advantage the proof steps we have established: according to step(\ref{r.uniquessEstimate}) and (\ref{r.w2}),  \begin{verse}  \itshape         %%变为斜体
    Any solution of Eq.(\ref{R.U.eq}), such as $W(x)$, it is reasonable to define
    \begin{equation*}
         W_{2} =\mathcal{E}(PW+B),\qquad\quad  W_{1}=W \cdot W_{2}
    \end{equation*}

    It follows that   $W_{2}$ is bounded ,
        \begin{equation}
              {\frac {d}{dx}} W_{1}=A W_{1}
        \end{equation}
    and
     \begin{equation}
        W_{1} =\mathcal{E} [A  ] \cdot  W_{1}|_{x=0}= \mathcal{E} [A ] \cdot W|_{x=0}=0
     \end{equation}
    Therefore, $W(x)=W_{1}\cdot W_{2}^{-1}=0$
\end{verse}

}

\end{enumerate}
\end{proof}

%%%%%%%%%%%%%%%%%%%%%%%%%%%%%%%%%%%%%%%%%%%%%%%%%%%%%%%%%%%%%%%%%%%%%%%%%%%%%%%%%%%%%%%%%%%%%%%%%%%%%%%%%%%%%%%%
%%%%%%%%%%%%%%%%%%%%%%%%%%%%%%%%%%%%%%%%%%%%%%%%%%%%%%%%%%%%%%%%%%%%%%%%%%%%%%%%%%%%%%%%%%%%%%%%%%%%%%%%%%%%%%%%

\subsection{Simplify solutions of Riccati equation by particular solution}
In the research of Riccati equation, particular solution plays crucial important role. Too much of works have been done.
The first important result in the analysis of the Riccati equation is that if one solution is known then a whole family of solutions can be found \cite{SolutionOfRiccati}.
\begin{thm}
The same conditions as theorem \ref{th.Riccati}, Riccati equation
\begin{equation}  \label{r.ps}
    {\frac {d}{dx}}W+WPW+WB-AW-Q=0
\end{equation}
 has the unique  solution
    \begin{equation}
       \begin{aligned}
	 W =Y+\mathcal{E}\big (A-YP\big )\cdot\big( W|_{x=0}\big)
                         \cdot\biggl[I+\int _{0}^{x}\!R(t) \left( t \right) {dt}\cdot ( W\mid_{x=0})\biggl]^{-1}\cdot
                          \mathcal{F}\big (-[B+PY]\big )
       \end{aligned}
   \end{equation}
            where
                \begin{equation}
                  \left\{
                   \begin{aligned}
                          Y& \hbox{ is solution of Eq.(\ref{r.ps}) when } W|_{x=0}=0,
                           \hbox{i.e. } Y|_{x=0}=0\\
                          R&:=\mathcal{F}\big (-[B+PY]\big ) \cdot P\cdot\mathcal{E}\big (A-YP\big )\\
                   \end{aligned}
                   \right.
              \end{equation}

\end{thm}

%%%%%%%%%%%%%%%%%%%%%%%%%%%%%%%%%%%%%%%%%%%%%
%%%%%%%%%%%%%%%%%%%%%%%%%%%%%%%%%%%%%%%%%%%%%
\begin{proof}
\begin{enumerate}
  \item {According to Theorem \ref{th.Riccati} , Eq.(\ref{r.ps}) has solutions. Take any one of it, such as $Y$, and  let
       \begin{equation}\label{simpleStep1}
            V=W-Y
        \end{equation}
        It follows that
        \begin{equation}\label{simpleStep2}
        \begin{array}{ll}
               VPV&=(W-Y)P(W-Y)
               =\big (WPW-YPY\big ) -(W-Y)PY
                 -YP(W-Y)
                 =\big(WPW-YPY\big) -VPY-YPV\\
                  &
                  \stackrel{{Eq.(\ref{r.ps})} }{=}
                  \Big([-  {\frac {d}{dx}}W -WB+AW+Q]
                  -[-  {\frac {d}{dx}}Y -YB+AY+Q] \Big) -VPY-YPV\\
                  &= \Big(-  {\frac {d}{dx}}V +AV-VB   \Big) -VPY-YPV
                  =-  {\frac {d}{dx}}V +(A-YP)V-V(B+PY)
        \end{array}
        \end{equation}
        That is,
        \begin{equation}    \label{r.psolution.V}
            {\frac {d}{dx}}V +VPV+V(B+PY)-(A-YP)V=0
        \end{equation}
  }
  \item {Obviously, $\mathcal{E} \Big(A-YP \Big)$ and $\mathcal{F}\Big(-[B+PY]\Big)$ are reversible , we may let
      \begin{equation}\label{r.ps.transform}
            V=\mathcal{E}\Big(A-YP\Big)\cdot U \cdot \mathcal{F}\Big(-[B+PY]\Big)
       \end{equation}
      Now Eq.(\ref{r.psolution.V}) could be transformed into
      \begin{multline}
         \biggl[ \mathcal{E} \Big(A-YP \Big)\cdot {\frac {d}{dx}}U \cdot\mathcal{F} \Big(-[B+PY] \Big)
         +(A-YP)V              -V(B+PY) \biggl]              +VPV
             +V(B+PY)-(A-YP)V=0
      \end{multline}

     or,
     \begin{equation}
           {\frac {d}{dx}}U+U\biggl[\mathcal{F}\Big(-[B+PY]\Big) \cdot P\cdot\mathcal{E}\Big(A-YP\Big)\cdot  \biggl]U=0
       \end{equation}
      }
  \item {Let
        \begin{equation}
           R:=\mathcal{F}\Big(-[B+PY]\Big) \cdot P\cdot\mathcal{E}\Big(A-YP\Big)
       \end{equation}
       According to Theorem \ref{th.Riccati}, $U$ has solution
       \begin{equation}
           U= W_{1}\cdot W^{-1}_{2}
       \end{equation}
       where
    \begin{equation}
       \begin{aligned}
         \left[ \begin {array}{c} W_{{1}}\\ \noalign{\medskip}W_{{2}}\end {array} \right]
         =\mathcal{E}\biggl(\left[ \begin {array}{cc} 0&0\\ \noalign{\medskip}R&0\end {array}
        \right] \biggl)\cdot\left[ \begin {array}{c} U\mid_{x=0}\\ \noalign{\medskip}I\end {array} \right]
        = \biggl(I + \int _{0}^{x}\! { \left[ \begin {array}{cc} 0&0\\ \noalign{\medskip}R&0\end {array}
        \right] {dt}\biggl)}  \cdot\left[ \begin {array}{c} U\mid_{x=0}\\ \noalign{\medskip}I\end {array} \right]
        =\left[ \begin {array}{c} U\mid_{x=0}\\ \noalign{\medskip}   I+\int _{0}^{x}\!R(t) \left( t \right) {dt}
                \cdot U\mid_{x=0}
        \end {array} \right]
       \end{aligned}
   \end{equation}

    Now, Let's consider how to choose Y , so that both $W$ and $U\mid_{x=0}$ are as simple as possible. It's clear that
    \begin{center}
        when $Y|_{x=0}=0$,   $U|_{x=0}=Y|_{x=0}=W|_{x=0}$
    \end{center}
    In this case,
    \begin{equation}
           U= W|_{x=0}\cdot \biggl[I+\int _{0}^{x}\!R(t) \left( t \right) {dt}\cdot ( W\mid_{x=0})\biggl]^{-1}
     \end{equation}\\\\
    It should be noticed that $\biggl[I+\int _{0}^{x}\!R(t) \left( t \right) {dt}\cdot ( W\mid_{x=0})\biggl]$  is reversible, otherwise
    \begin{equation}
           I+\int _{0}^{x}\!R(t) \left( t \right) {dt}\cdot ( W\mid_{x=0})\equiv 0
     \end{equation}
    which is clearly impossible.\\
  }
  \end{enumerate}

 According to   transformation(\ref{r.ps.transform}), the solution of Eq.(\ref{r.ps}) is
    \begin{equation}
       \begin{aligned}
	 W  =Y  + V
        =Y  +\mathcal{E}\Big(A-YP\Big)\cdot W|_{x=0}
                         \cdot   \biggl[I+\int _{0}^{x}\!R(t) \left( t \right) {dt}
                         \cdot ( W\mid_{x=0})\biggl]^{-1}
                           \cdot \mathcal{F}\Big(-[B+PY]\Big)
       \end{aligned}
   \end{equation}
   \emph{ where $Y(x)\equiv 0$ , if and only if $Q(x)\equiv 0$.}

\end{proof}

%%%%%%%%%%%%%%%%%%%%%%%%%%%%%%%%%%%%%%%%%%%%%%%%%%%%

\section{Acknowledgments}
Thanks Prof.Qiyan Shi's enthusiastic instruction and precious advice on
the thesis . The work is also supported by Prof.Youdong Zeng; thanks for
his many helpful discussions and suggestions on this paper. Besides, thanks
Prof.Guowei Chen for many valuable personal communications and guidance
concerning the school work.

\bibliographystyle{unsrt}%按引用顺序排列
%\bibliographystyle{abbrv}
%\bibliographystyle{plainnat}

% \bibliography{bibs}
%% Authors are advised to submit their bibtex database files. They are
%% requested to list a bibtex style file in the manuscript if they do
%% not want to use model5-names.bst.

%% References without bibTeX database:

% \begin{thebibliography}{00}

%% \bibitem must have one of the following forms:
%%   \bibitem[Jones et al.(1990)]{key}...
%%   \bibitem[Jones et al.(1990)Jones, Baker, and Williams]{key}...
%%   \bibitem[Jones et al., 1990]{key}...
%%   \bibitem[\protect\citeauthoryear{Jones, Baker, and Williams}{Jones
%%       et al.}{1990}]{key}...
%%   \bibitem[\protect\citeauthoryear{Jones et al.}{1990}]{key}...
%%   \bibitem[\protect\astroncite{Jones et al.}{1990}]{key}...
%%   \bibitem[\protect\citename{Jones et al., }1990]{key}...
%%   \harvarditem[Jones et al.]{Jones, Baker, and Williams}{1990}{key}...
%%

% \bibitem[ ()]{}

% \end{thebibliography}

\end{document}